\documentclass[12pt]{article}
\usepackage{amsthm, amsmath, amssymb, amsfonts, url, booktabs, tikz, setspace, fancyhdr, bm}
\usepackage{hyperref}
\usepackage{geometry}
\usepackage{hyperref, enumerate}
\usepackage[shortlabels]{enumitem}
\usepackage[babel]{microtype}
\usepackage[english]{babel}
\usepackage[capitalise]{cleveref}

\usepackage{algorithm}
\usepackage{algorithmic}
\usepackage{bbm,tkz-graph,subcaption}
\usepackage{csquotes}
\usepackage{mathrsfs}
\usepackage{mathabx}
\usetikzlibrary{patterns}
\usetikzlibrary{shapes} 
\usepackage{bbm}

\usepackage[utf8]{inputenc} 
\usepackage[T1]{fontenc}
\usepackage{amsfonts}
\usepackage{amscd}
\usepackage{graphicx}
\usepackage{enumitem}
\usepackage{verbatim}
\usepackage{hyperref}
\usepackage{amsmath,caption}
\usepackage{url,pdfpages,xcolor,framed,color}
\usepackage{todonotes}

\newtheorem{thr}{Theorem}[section]

\newtheorem{prop}[thr]{Proposition}
\newtheorem{conj}[thr]{Conjecture}
\theoremstyle{definition}

\newtheorem*{defi*}{Definition}

\newtheorem{claim}[thr]{Claim}

\def\G{\mathcal{G}}

\def\H{\mathcal{H}}

\newcommand*{\myproofname}{Proof}
\newenvironment{claimproof}[1][\myproofname]{\begin{proof}[#1]}{\end{proof}}

\newcommand*{\abs}[1]{\left\lvert #1\right\rvert}

\def\ct{\tilde{c}}

\newcommand{\ex}{ex}

\title{The maximum number of connected sets in regular graphs}

\date{}
\author{
Stijn Cambie \thanks{Department of Computer Science, KU Leuven Campus Kulak-Kortrijk, 8500 Kortrijk, Belgium. 
\protect\href{mailto:stijn.cambie@hotmail.com}{\protect\nolinkurl{stijn.cambie@hotmail.com}},
\protect\href{mailto:jan.goedgebeur@kuleuven.be}{\protect\nolinkurl{jan.goedgebeur@kuleuven.be}} and \protect\href{mailto:jorik.jooken@kuleuven.be}{\protect\nolinkurl{jorik.jooken@kuleuven.be}}}
\and Jan Goedgebeur
 \footnotemark[1] \thanks{Department of Applied Mathematics, Computer Science and Statistics, Ghent University, 9000 Ghent, Belgium.}
\and Jorik Jooken
 \footnotemark[1]
}

\begin{document}
\maketitle
\begin{abstract}
    We improve the best known lower bounds on the exponential behavior of the maximum of the number of connected sets, $N(G)$, and dominating connected sets, $N_{dom}(G)$, for regular graphs. These lower bounds are improved by constructing a family of graphs defined in terms of a small base graph (a Moore graph), using a combinatorial reduction of these graphs to rectangular boards followed by using linear algebra to show that the lower bound is related to the largest eigenvalue of a coefficient matrix associated with the base graph. We also determine the exact maxima of $N(G)$ and $N_{dom}(G)$ for cubic and quartic graphs of small order. 
    We give multiple results in favor of a conjecture that each Moore graph $M$ maximizes the base indicating the exponential behavior of the number of connected vertex subsets among graphs with at least $\abs M$ vertices and the same regularity.
    We improve the best known upper bounds for $N(G)$ and $N_{dom}(G)$ conditional on this conjecture.
\end{abstract}

\section{Introduction}

Connectedness of graphs is a fundamental property related to the possibility of interaction and transmission, and as such one can expect that the robustness of a network is related to the number of connected sets.
There are also relations between e.g. the number of connected sets and the average size of a connected set (\cite[Thr.~3.2]{John22b}).
The number of connected (vertex) sets in a graph, is related to upper bounds for the time complexity of certain hard algorithmic questions for which no subexponential time algorithm is known at this point. 
Some examples of such instances are the travelling salesman problem, computing the chromatic - or Tutte polynomial, finding a maximum internal spanning tree, optimal Bayesian network, Hamiltonian path or - cycle.

An induced subgraph of $G=(V,E)$ which is induced by a subset of vertices $V'$ of $V$ is denoted by $G[V']=\left(V',E\cap \binom{V'}{2}\right).$
The disjoint union of two graphs $G=(V,E)$ and $H=(V',E')$, denoted by $G+H$, is the graph defined by $(V \cup V', E \cup E').$ We will also write $2G$ for $G+G.$
The edge union of $G=(V,E)$ and $H=(V',E')$, where $V'\subset V$, denoted by $G \cup H$, equals $(V, E \cup E').$ In this paper, $n$ always denotes the order of a graph.

For a graph $G$, $i(G)$, $N(G)$, and $N_{dom}(G)$ denote, respectively, the number of independent sets, the number of vertex subsets $V' \subseteq V(G)$ for which $G[V']$ is connected, and the number of vertex subsets $V' \subseteq V(G)$ for which $V'$ is a dominating set in $G$ and $G[V']$ is connected. We will also refer to such vertex sets as \textit{connected sets} and \textit{connected dominating sets}. For a graph $G$ of order $n$, let $c(G)= \sqrt[n]{N(G)}$ and $\ct(G)= \sqrt[n]{N_{dom}(G)}$. 
Let $\G_{d,g}$ be the set of $d$-regular graphs with girth at least $g$ and $\G_{d}=\G_{d,3}$. Let $\G_{d,g}(n) \subset \G_{d,g}$ and $\G_{d}(n) \subset \G_{d}$ be the subsets restricted to the graphs of order $n$.
Let $c_{d,g}(n)=\max_{G \in \G_{d,g}(n)} c(G)$ and $c_{d}(n) =c_{d,3}(n)$. Let $\ct_{d,g}(n)=\max_{G \in \G_{d,g}(n)} \ct(G)$ and $\ct_{d}(n) = \ct_{d,3}(n)$. Let $c_d= \limsup_n c_d(n)$ and $\ct_d= \limsup_n \ct_d(n).$

We also use some Landau notation.
The statement $f(x)=O(g(x))$ implies that there exist fixed constants $x_0, M>0$, such that for all $x \ge x_0$ we have $\lvert f(x) \rvert \le M \lvert g(x) \rvert .$
We write $f(x)=\Theta(g(x))$ if $g(x)=O(f(x))$ and $f(x)=O(g(x))$.
When $\lim_{ x \to \infty} \frac{ f(x)}{g(x)}=0$, we write $f(x)=o(g(x)).$

Determining the number of connected sets $N(G)$ in a graph $G$, is a hard task. E.g. even for elementary constructions such as the grid $P_n \times P_n$ and the cube $Q_n=K_2^n$ no general formulas are known, see~\cite{OEISA059525, OEISA290758} and~\cite[Ques.~2,3]{Vince20}. 
In this paper, we study the maximum number of connected sets in $d$-regular graphs. A maximum degree condition is natural to avoid too large (and trivial) bounds attained by the complete graph $K_n$ or graphs containing the star $S_n$ as a subgraph.
Since the number of connected sets increases by edge addition, the extremal graphs are edge-maximal (under the maximum degree condition) and hence (almost) regular. It is natural to believe that the maximum is attained by regular graphs whenever they exist (i.e., whenever $2 \mid dn$). The main parameter which we will study in this paper is $c_d$, the base of the exponential behavior of the maximum possible number of connected sets in $d$-regular graphs.\footnote{Similar to~\cite{Bjorklund12, KKK18}, but in contrast to~\cite{John22b,Vince20} who normalized with a factor $\frac 12$.}
It turns out that $c_d$ was well-estimated experimentally in~\cite{Perrier08} for $d \in \{3,4,5\}$. 
Bj\"orklund et al.~\cite{Bjorklund12} and Kangas et al.~\cite{KKK18} determined upper bounds using the product theorem of~\cite{CGFS86}, by considering the first and second neighborhood around every vertex.
In particular, this implies that for every fixed $d$, $c_d<2$.
Lower bounds were determined by Kangas et al.~\cite{KKK18} using a modified version of the ladder graph, i.e. chaining certain subgraphs (called gadgets), as depicted in~\cref{fig:gen_laddergraphs}.
Later, the study of the number of connected sets in a graph was reinitiated by Vince~\cite{Vince20} and Haslegrave~\cite{John22b}, motivated by the study of the average size of a connected induced subgraph which originated from the work of~\cite{jamison_average_1983} and~\cite{KMO18}.
As such Haslegrave's question from~\cite{John22b} if the criss-cross prism (depicted in Fig.~\ref{fig:gen_laddergraphs}) is asymptotically optimal (i.e., $c_3=\sqrt[4]{7}$) was already addressed in earlier work. 

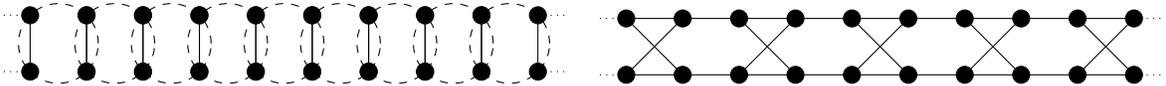
\begin{figure}[h]
    \centering
    \begin{tikzpicture}[scale=0.75]
    \draw[dotted] (0,0)--(-0.5,0);
\draw[dotted] (0,1)--(-0.5,1);
\draw[dotted] (9,0)--(9.5,0);
\draw[dotted] (9,1)--(9.5,1);
    \foreach \x in {0,1,...,8}{
        \draw[dashed] (\x+0.5,0.5) circle [radius=0.707]; 
         \draw[fill] (\x,0) circle (0.15);
        \draw[fill] (\x,1) circle (0.15);
        \draw[fill] (\x+1,0) circle (0.15);
        \draw[fill] (\x+1,1) circle (0.15);
        \draw (\x+1,0)--(\x+1,1);
        \draw (\x,0)--(\x,1);
    }
    \end{tikzpicture}\quad
    \begin{tikzpicture}[scale=0.75]

\draw[dotted] (0,0)--(-0.5,0);
\draw[dotted] (0,1)--(-0.5,1);
\draw[dotted] (9,0)--(9.5,0);
\draw[dotted] (9,1)--(9.5,1);
\draw (0,0)--(9,0);
\draw (0,1)--(9,1);
\foreach \x in {0,1,...,9}{
\draw[fill] (\x,0) circle (0.15);
\draw[fill] (\x,1) circle (0.15);
}
\foreach \x in {0,2,4,6,8}{
\draw (\x,0)--(\x+1,1);
\draw (\x,1)--(\x+1,0);
}

    \end{tikzpicture}    
    \caption{The generalized ladder and the criss-cross prism}
    \label{fig:gen_laddergraphs}
\end{figure}

We improve the lower bounds on $c_d$ by considering constructions which are very different from the chained gadgets (modified ladders).
We do so by connecting base graphs along a cycle, resulting in a few structured families which can be analyzed by considering the main term of $N(G)$ after a combinatorial reduction towards rectangular boards.
A connected subgraph of our graph, will now correspond with a so-called tile, a connected subset of cells of the board.
Here we can approximate and count the connected (defined later) subsets, called tiles, using a recursion. The solution of the recursion gives rise to a linear combination of exponential functions, the main term being related to the largest eigenvalue of the related coefficient matrix.
Our new constructions depend on smaller ``gadgets''/ base graphs than the ones used in~\cite{KKK18}, while improving the constants.
For small orders, we also compute\footnote{We refer the interested reader to~\cref{sec:app1} for more details on the computer programs. The code and data related to this paper have been made publicly available at~\cite{CGJ23}.} the extremal graphs, which turn out to be mostly well-structured and known graphs.
Nevertheless, in the general case, we tend to believe that the extremal graphs cannot be described exactly.
Despite that, we propose two conjectures.

The first one is an analogue of~\cite[Conj.~17]{CdJdVK23}. See~\cite{CD22} and references therein for the main related problem. 
It essentially says that the complement graph needs to have as many as possible cliques $K_{n-d}$. Let $i(G)$, $N(G)$ and $N_{dom}(G)$ denote respectively the number of independent sets, the number of connected sets and the number of dominating connected sets of the graph $G.$

\begin{conj}\label{conj:maxgiven_nd}
    Given $d$ and $d+1 \le n \le 2d.$ Among all $d$-regular graphs with order $n$, for each of the quantities $i(G)$, $N(G)$ and $N_{dom}(G)$, the graph maximizing the quantity is the complement of one or more copies of $K_{n-d}$ and a graph $H$, where $\abs H <2(n-d).$ Here $H$ can be different depending on the quantity.
\end{conj}

Second, even while the exact extremal graphs are hard to describe in general and the sequence $c_d(n)$ is not monotone decreasing (as we prove later, see~\cref{prop:cd_notmono_d4} and~\cref{prop:nonmonotone}), we conjecture that the Moore graphs are extremal and give upper bounds on $c_d$. Moore graphs are $d$-regular graphs of a given girth $g$ which attain the theoretical lower bound (the Moore bound) on the order. 

\begin{conj}\label{conj:Mooregraphs_extremal}
    If there exists a $d$-regular Moore graph $G$ of order $n$, then for every $n'>n$,
    $c_d(n')<c_d(n)=c_d(G)$.
\end{conj}

Once the latter conjecture is proven, this would result in better upper bounds for $c_d.$ 
The best lower- and upper bounds from different papers, rounded up to $3$ decimals, are given in~\cref{tab:overview_c3_c4}, the upper bounds depending on~\cref{conj:Mooregraphs_extremal} are in italics and red.

\begin{table}[h]
    \centering
    \begin{tabular}{|c|c c | c c |}
\hline
& $c_3 \ge$ & $c_3 \le$ & $c_4 \ge$ & $c_4 \le$ \\
\hline
\cite{Bjorklund12} &  & 1.968 & & 1.987 \\
\cite{KKK18} & 1.765 & 1.935 & 1.893 & 1.981  \\
\cite{Vince20} & 1.554 & & & \\
\cite{John22b} & 1.627 & 1.956  & & \\ 
This work & 1.792 &\textit{ \textcolor{red}{1.821} }& 1.897 & \textit{\textcolor{red}{1.932}} \\
\hline
\end{tabular}
    \caption{Bounds on $c_3$ and $c_4$}
    \label{tab:overview_c3_c4}
\end{table}

We can also compare the lower bounds for $\ct_d$ with the work of~\cite{Bjorklund12} (see~\cref{tab:overview_ct345}). We remark that our improved constructions use smaller gadgets than those used in~\cite{Bjorklund12}.

\begin{table}[h]
    \centering
    \begin{tabular}{|c|c c   c |}
\hline
& $\ct_3 \ge$  & $\ct_4 \ge$ &  $\ct_5 \ge$ \\
\hline
\cite{Bjorklund12} & 1.648   & 1.838 & 1.923 \\
This work & 1.743   & 1.875 & 1.940 \\
\hline
\end{tabular}
    \caption{Lower bounds on $\ct_d$ for $d \in \{3,4,5\}$}
    \label{tab:overview_ct345}
\end{table}

\subsection{Outline of the paper}\label{subsec:outline}

We derive some fundamental results and make some additional observations in~\cref{sec:el_obs}.
In~\cref{sec:cubic_fam1}, we prove a lower bound for the number of connected sets in cubic graphs by considering an elementary family.
This family is described in a general form in~\cref{sec:lowerbounds}. We derive improved lower bounds for $c_d$ and $\ct_d$ by considering these families.
\cref{sec:d34_smallorder} contains the graphs (sometimes among the graphs with an additional girth condition) maximizing $c_d(n)$ and $\ct_d(n)$ for $d \in \{3,4\}$ and small $n.$
Finally, in the conclusion,~\cref{sec:conc}, we summarize our results, and add some remarks on related quantities.
In the appendix, we briefly explain the computer programs used for computations in~\cref{sec:lowerbounds} and~\cref{sec:d34_smallorder}.

\section{Fundamental results}\label{sec:el_obs}

In this section, we prove multiple fundamental insights related to the number of connected sets in regular graphs.
We first prove that the maximum number of connected subsets is (as expected) attained by a connected graph.
\begin{prop}\label{prop:extremal_is_connected}
    For every disconnected $d$-regular graph $G$ of order $n$, $c(G)<c_d(n).$
\end{prop}

\begin{proof}
    For $d=2,$ it is sufficient to note that $N(C_a + C_b)= N(C_a)+N(C_b)<N(C_{a+b})$ for every $a,b \ge 3.$
    Note that $N(C_{a})=a^2-a+1$, as for every length $1 \le i \le a-1$ there are $a$ paths of length $i$, and also the whole cycle forms a connected vertex subset.
    Thus $N(C_a)+N(C_b)<N(C_{a+b})$ as $1<2ab.$

    For $d\ge 3,$ we can take two connected $d$-regular graphs $G_1$ and $G_2$. Let $uv$ and $wx$ be an arbitrary edge of $G_1$ and $G_2$, respectively.
    Now by deleting the latter two edges and adding the edges $uw, vx$, we obtain a graph $G_3$ of order $\abs{G_1}+ \abs{G_2}$ which is $d$-regular (see~\cref{fig:connecting2graphs}).
    Now we have that $N(G_3)>N(G_1)+N(G_2).$ 
    Let $N(G_1, uv)$ and $\overline N(G_1, uv)$ be the number of connected vertex subsets containing both $u$ and $v$, respectively not both of $u$ and $v$. The quantities $N(G_2, wx)$ and $\overline N(G_2, wx)$ are defined similarly.
    Without loss of generality, we can assume that $N(G_1, uv) \ge N(G_2, wx).$
    Then $N(G_3) > \overline N(G_1, uv)+ \overline N(G_2, wx)+ 2 N(G_1, uv) \ge N(G_1)+N(G_2).$ 
    The first inequality is true since any connected vertex subset of $G_1$ containing $u,v$ can be extended with a shortest path from $w$ to $x$ in $G_2$, to form a connected vertex subset in $G_3$. Furthermore, we can extend the shortest path in $G_2$ with a vertex from $G_2$ not on the shortest path (so there are at least two ways).
    Any connected vertex subset in $G_1$ not containing both $u$ and $v$, is also connected within $G_3.$
    The connected subsets (edges) $uw$ and $vx$ have not been counted yet, implying that the difference was strict.
    The second inequality is by the assumption that $N(G_1, uv) \ge N(G_2, wx).$
    We conclude that $c_d(n)$ is not attained by a disconnected graph.
\end{proof}

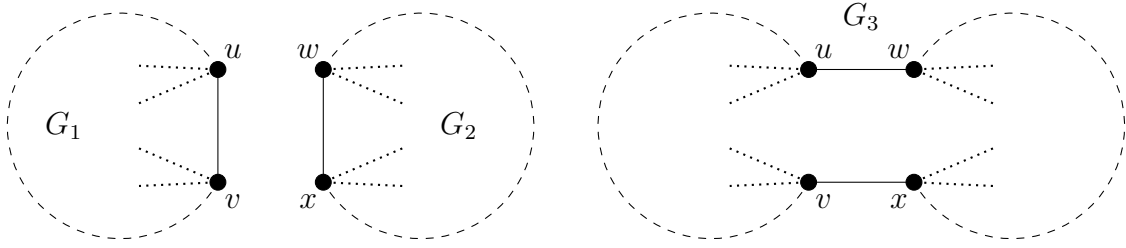
\begin{figure}[h]
    \centering
    \begin{tikzpicture}
         \draw[dashed] (30:1.5) arc (30:330:1.5);
         \draw[dashed] (2.7,0.75) arc (150:-150:1.5);
        \draw[fill=black] (30:1.5) circle (3pt);
        \draw[fill=black] (-30:1.5) circle (3pt);
        \draw[fill=black] (2.7,0.75) circle (3pt);
        \draw[fill=black] (2.7,-0.75) circle (3pt);
        \node (G1) at (-0.75,0.) {$G_1$} ;
        \node (G2) at (4.5,0.) {$G_2$} ;
        \node (w) at (2.5,1) {$w$} ;
        \node (x) at (2.5,-1) {$x$} ;
        \node (u) at (1.5,1) {$u$} ;
        \node (v) at (1.5,-1) {$v$} ;
        \draw (2.7,0.75)--(2.7,-0.75);
        \draw (30:1.5) --(-30:1.5) ;
        \draw[thick, dotted] (0.25,0.3)--(30:1.5) --(0.25,0.8);
        \draw[thick, dotted] (0.25,-0.3)--(-30:1.5) --(0.25,-0.8);
        \draw[thick, dotted] (3.75,-0.3)--(2.7,-0.75) --(3.75,-0.8);
        \draw[thick, dotted] (3.75,0.3)--(2.7,0.75) --(3.75,0.8);
    \end{tikzpicture}\quad
    \begin{tikzpicture}
         \draw[dashed] (30:1.5) arc (30:330:1.5);
         \draw[dashed] (2.7,0.75) arc (150:-150:1.5);
        \draw[fill=black] (30:1.5) circle (3pt);
        \draw[fill=black] (-30:1.5) circle (3pt);
        \draw[fill=black] (2.7,0.75) circle (3pt);
        \draw[fill=black] (2.7,-0.75) circle (3pt);
       \node (G3) at (2,1.45) {$G_3$} ;
        \node (w) at (2.5,1) {$w$} ;
        \node (x) at (2.5,-1) {$x$} ;
        \node (u) at (1.5,1) {$u$} ;
        \node (v) at (1.5,-1) {$v$} ;
        \draw (2.7,0.75)--(30:1.5);
        \draw  (2.7,-0.75)--(-30:1.5) ;
        \draw[thick, dotted] (0.25,0.3)--(30:1.5) --(0.25,0.8);
        \draw[thick, dotted] (0.25,-0.3)--(-30:1.5) --(0.25,-0.8);
        \draw[thick, dotted] (3.75,-0.3)--(2.7,-0.75) --(3.75,-0.8);
        \draw[thick, dotted] (3.75,0.3)--(2.7,0.75) --(3.75,0.8);
    \end{tikzpicture}
    \caption{$G_1 + G_2$ and $G_3$}
    \label{fig:connecting2graphs}
\end{figure}

Next, we prove that the number of small connected sets is maximized by graphs with large girth, as a consequence of $\ex(n,T,K_{1,d+1})$ for a fixed tree $T$ being attained by graphs with sufficiently large girth (see~\cite{AS15} for the introduction of this concept in extremal graph theory). 
Here $\ex(n,T,H)$ is the maximum possible number of copies of $T$ that can be found in an $H$-free graph of order $n$.

\begin{prop}
    Let $d>0$ be a fixed integer.
    For $n$ sufficiently large in terms of $k$, the number of connected vertex subsets of size $k$ in a graph of order $n$ and maximum degree bounded by $d$ is maximized by any graph for which the girth is at least $k.$
\end{prop}

\begin{proof}
    Let $T$ be a fixed tree of order $k$. Remember that any tree has a central vertex or central edge.
    Once the central vertex or central edge is mapped to a vertex or edge of a graph $G$, upper bounds on the number of copies of $T$ can be given in function of the maximum degree $d$.
    Furthermore, these are sharp when every degree is exactly equal to $d$ and the girth is at least $k$ (as then no copy of $T$ is double-counted).

    Now, since any connected $k$-vertex subset contains a spanning tree (at least one), the number of connected $k$-sets is upper bounded by $\sum_{ T \in \mathcal T_k} \ex(n,T,K_{1,d+1})$, where $\mathcal T_k$ is the set of all non-isomorphic trees of order $k$.
    
    When the girth is large, every connected $k$-set induces a unique tree and thus all upper bounds are sharp, leading to the result.
\end{proof}

Since independent sets are (except for singletons) not connected,~\cref{conj:maxgiven_nd} may be considered as a surprising conjecture, as it implies that certain (complete partite) graphs maximize $i(G)$ and $N(G)$ at the same time.
We will prove it for $d$-regular graphs of order $n$, for which $n-d \mid n.$

It is trivial that when $d+1 \mid n$, the union of $K_{d+1}$s maximizes the number of cliques in a $d$-regular graph of order $n$. 
Taking the complement graph, when $n-d\mid n$, $r= \frac{n}{n-d}$, a balanced complete $r$-partite graph maximizes the number of independent sets in such a graph.
We prove that in this case, the same is true when maximizing the number of connected vertex subsets.\footnote{In general, the extremal graphs for $i$ and $N$ are different, also for $n<2d$. E.g. we verified this for $n=13, d=8$.}

\begin{prop}

    Let $d\ge 2, n \ge d+1$, $n-d \mid n$ and $r= \frac{n}{n-d}.$
    Among all graphs of order $n$ and maximum degree bounded by $d$, the balanced complete $r$-partite graph is the unique graph maximizing $N$ and $N_{dom}.$ 
\end{prop}

\begin{proof}
    Consider any value $k \ge 2.$
    For every vertex $v$, there are at least $n-d-1$ vertices not connected with $v$ and hence at least $\binom{n-d-1}{k-1}$ vertex $k$-subsets containing $v$ which do not induce a connected subgraph.
    This implies that the number of vertex subsets of size $k$ which do not induce a connected subgraph, is at least 
    $$\frac{n}{k}\binom{n-d-1}{k-1}=r\binom{n-d}{k}.$$
    Summing over all $2 \le k \le n-d$, this leads to at least 
    $r(2^{n-d}-(n-d+1))+1=r2^{n-d}-n-r+1$ many non-connected vertex subsets.
    Equality is only possible if every non-connected vertex subset of size $k$ is counted $k$ times, i.e. is independent, and every vertex has degree exactly $d.$
    That is, if the complement graph is a union of copies of $K_{n-d}$, and thus the graph equals the balanced complete $r$-partite graph.

    For the dominating case, note that in a balanced complete $r$-partite graph, every edge is dominating. 
    On the other hand, except for $n=d+1,$ a singleton never can be dominating.
    We thus conclude that for a graph $G$ with $\Delta(G) \le d<n-1$ and $\abs G =n,$
    $$N(G) \le 2^n- r2^{n-d}+n+r-1 \mbox{ and } N_{dom}(G) \le 2^n-r2^{n-d}+r-1. \qedhere$$
\end{proof}

\begin{prop}\label{prop:cd_notmono_d4}
    The sequence $c_4(n):= \max N(G)^{1/n}$, where the maximum is taken over $4$-regular graphs of order $n$, is not decreasing.
\end{prop}

\begin{proof}
    There are only $2$ resp. $6$ $4$-regular graphs of order $7$ and $8$~\cite{Me99}.
    The extremal graphs turn out to be $K_{3,4}^{++}=K_{3,4} \cup 2K_2$ (the two edges added in such a way that the graph is $4$-regular), and $K_{4,4}.$
    We have that $N(K_{3,4}^{++})=(2^3-1)(2^4-1)+9=114$ and $N(K_{4,4})=(2^4-1)^2+8=233.$
    This implies that $N(K_{4,4})>2N(K_{3,4}^{++})$.
    Since $c_4(n)<2$ for every $n \in \mathbb N$, the conclusion is immediate.
\end{proof}

Conditional on~\cref{conj:maxgiven_nd}, the non-monotonicity can be proven for all regularities $d \ge 4.$

\begin{prop}\label{prop:nonmonotone}
    Conditional on~\cref{conj:maxgiven_nd}, the sequence $c_d(n):= \max N(G)^{1/n}$ where the maximum is taken over $d$-regular graphs of order $n$, is not monotone for $d \ge 4.$
    Moreover, the sequence $(c_d(n))_n$ can contain arbitrarily long increasing subsequences, for $d$ sufficiently large.
\end{prop}

\begin{proof}
    When $n=2d-r$, and $r<d/2$, by~\cref{conj:maxgiven_nd}, the $d$-regular graph $G$ of order $n$ maximizing $N(G)$ is the union of $K_{d,d-r}$ and a $r$-regular graph $H$ on the bipartition class of size $d.$
    Now $N(K_{d,d-r} \cup H)=(2^d-1)(2^{d-r}-1)+N(H)+d-r.$ 
    Here $H$ is connected by~\cref{prop:extremal_is_connected} and $N(H)=O(c_r^d)=o(2^d)$.
    Fix $k>0$. Note that for every $r<k,$ $c_r \le c_k.$
    Then for $d$ sufficiently large, $2^{ d-k} >>c_r^d$.
    For $d >>r,$ this implies that $c(G)= \sqrt[n]{1-2^{r-d}-2^{-d} +o(2^{r-d})} \sim 1-\frac{2^r+1}{2^d(2d-r)}.$
    Since the latter is a decreasing function as $r$ increases, $(c_d(n))_n$ is increasing for $2d-r \le n \le 2d.$
    For $n=2d-1,$ the complement has to be $K_{d-1}$ and $K_d \cup M$, i.e. $G$ is $K_{d,d-1}$ with an additional matching $M$ of $\frac d2$ edges in the larger bipartition class.
    In this case, $N(G)=(2^d-1)(2^{d-1}-1)+5d/2-1$.
    Similarly, for $n=2d-2$, $G$ equals $K_{d,d-2}\cup C_d$, where $C_d$ is a cycle added on the larger bipartition class of $K_{d,d-2}.$
    In that case $N(G)=(2^d-1)(2^{d-2}-1)+d^2-1$.
    By comparing these values with $N(K_{d,d}),$ we conclude that for every $d>3$, the sequence $c_d(n)$ is not monotone.
    More precisely, one can check that 
   $$ (2^d-1)(2^{d}-4)+4d^2-4 < (2^d-1)(2^{d}-2)+5d-2< (2^d -1)^2+2d. \qedhere$$\end{proof}

\section{A family of cubic graphs with many connected vertex subsets}\label{sec:cubic_fam1}

In this section we define an elementary family of graphs, which which will be extended in~\cref{sec:lowerbounds}. Let $k \geq 1$ be an integer. Let $G$ be the graph $C_{3k}+kP_1$ (on $4k$ vertices) with some additional edges (defined hereafter). Label the vertices of $C_{3k}$ as $ 1,2, \ldots, 3k $ in order, and the vertices of $kP_1$ as $3k+1, 3k+2,\ldots, 4k$.
For each $i \in [k]$, connect the vertex $3k+i$ with $i, k+i$ and $2k+i.$

We can represent vertex subsets of $G$ with cells in a $4 \times k$-board, where cell $(i,j), i \in [4],j\in [k]$ represents the vertex $k(i-1)+j.$ 
We will call two cells adjacent, if the vertices in $G$ they represent are adjacent.
A subset of cells of the board are connected if the corresponding vertices induce a connected subgraph of $G$.
We call a connected subset of cells of the board a tile.
Note that adjacent cells belong to consecutive columns (we consider column $1$ and $k$ at this point as adjacent as well, i.e. consecutive modulo $k$) or the same column. 
For ease of thinking and understanding the connectedness of the board, we mention all pairs of cells that are adjacent (which is a symmetric property).
A cell in one of the first three rows is adjacent to its horizontal neighbor(s), as well as to the lowest cell in the same column.
The connections between the cells in the first and last column are precisely as follows: $(1,1)$ is adjacent with $(3,k)$, $(2,1)$ is adjacent with $(1,k)$ and $(3,1)$ is adjacent with $(2,k)$.

\begin{figure}[h]
    \centering
        \begin{tikzpicture}[scale=1.15]
        \fill[gray!70!white] (0,2) rectangle (1,4);
        \fill[gray!70!white] (0,0) rectangle (1,1);
        \fill[gray!70!white] (0,3) rectangle (4,4);
        \fill[gray!70!white] (2,0) rectangle (3,1);
        \draw[step=1cm] (0,0) grid (4,4);
    \end{tikzpicture}\quad
    \begin{tikzpicture}[scale=0.52]
\foreach \x in {0,1,...,12}{
\draw[fill] (\x*360/12:4.5) circle (0.15);
\draw (\x*360/12+360/12:4.5)--(\x*360/12:4.5);
}
\foreach \x in {1,4,7,10}{
    \draw[fill] (\x*30+15:1.5) circle (0.15);
    \foreach \y in {0,4,8}{
        \draw (\x*30+30*\y:4.5)-- (\x*30+15:1.5) ;
    }

}

\foreach \x in {0,1,...,3}{
\draw[fill, blue] (\x*360/12:4.5) circle (0.15);
\draw[blue] (\x*360/12+360/12:4.5)--(\x*360/12:4.5);
}
\draw[fill, blue] (4*360/12:4.5) circle (0.15);
\draw[blue] (0:4.5)--(135:1.5)--(120:4.5);
\draw[blue] (60:4.5)--(315:1.5);
\foreach \x in {4,10}{
\draw[fill, blue] (\x*30+15:1.5) circle (0.15);
}
\end{tikzpicture}
    \caption{Example of a tile for the graph $G$ when $k=4$, and the corresponding subgraph in $G$}
    \label{fig:enter-label}
\end{figure}
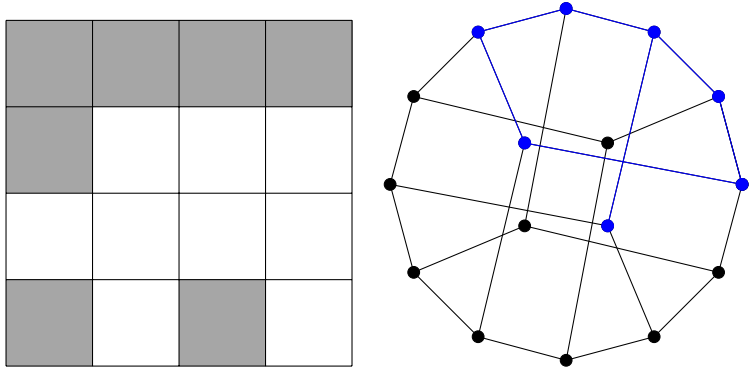

A path of ordered vertices $v_1v_2 \ldots v_{\ell}$ is defined to \textit{go to the right} if for every $r \in [\ell-1]$, it satisfies either $v_{r+1} \equiv v_{r} \pmod k$ or $v_{r+1} \equiv v_{r}+1 \pmod k$.

Similarly, a path of vertices $v_1v_2 \ldots v_{\ell}$ is said to \textit{go to the left} if for every $r \in [\ell-1]$, either $v_r \equiv v_{r+1} \pmod k$ or $v_r \equiv v_{r+1}+1 \pmod k$.

The corresponding definitions for cycles are analogous, where $v_0=v_{\ell}$, and all other pairs of vertices are distinct.

We first prove that it is essentially sufficient to count the number of subsets of the board which are \textit{intensely} connected. 
These are the connected subsets for which every two consecutive columns contain a pair of adjacent cells (between the columns).
For ease of counting, we will assume that the four cells in the first column are present.

Let $\H_1$ be the family of tiles for which not all columns are non-empty.
Let $\H_2$ be the family of tiles for which all columns are non-empty.
Let $\H_2^\bullet \subset \H_2$ be the subfamily containing all tiles that contain all cells of the first column.
Let $\overline \H_2^\bullet \subset \H_2^\bullet$ be the subfamily containing all tiles containing a cycle which goes to the right (equivalently to the left), passing through every column.


\begin{prop}\label{prop:afschattingenH1etc}
We have a) $\abs{\H_1} \le n(n-1) \abs{\H_2}$, b) $\abs{\H_2} \le 15 \abs{\H_2^\bullet}$ and c) $\abs{\H_2^\bullet} \le 15n  \abs{\overline \H_2^\bullet}$.
\end{prop}

\begin{proof}
    For the first inequality, note that for a tile in $\H_1,$ there are a number $i$ of consecutive columns which are empty, starting with the one of index $j$. By replacing all of these with full columns, we end up with a tile in $\H_2 $. 
    Furthermore, for every two different tiles with the same pair $(i,j)$ associated, the result of the operation is a different subset.
    This implies that $\abs{\H_1} \le n(n-1)  \abs{  \H_2  }.$

    The second inequality is almost immediate. Adding the non-present cells of the first column, results into a subset with the desired inequalities. At most $15$ different tiles in $\H_2$ can result into the same tile by adding such vertices.

    To finish, we prove the third inequality.
    Take a tile in $\H_2^\bullet$ which does not belong to $\overline \H_2^\bullet.$ 
    For every cell (not in the first column), there is one direction, such that there is a path connecting the cell with the cells in the first column in this direction (left or right).
    Let column $i$ be the most left (the one with smallest index) which contains a cell which is only connected with the first column with a path going to the right.
    Add the missing vertices in this column. We end up with a tile in $\overline \H_2^\bullet$.
    For every possible completed column, there are at most $15$ initial possible configurations.
    Since there are at most $n$ indices for that column, the inequality is clear. 
\end{proof}

From~\cref{prop:afschattingenH1etc}, we conclude that $\abs{\H_1}+\abs{\H_2}<225n^3 \abs{\overline \H_2^\bullet}$. 
The counting problem for the number of tiles in $\overline \H_2^\bullet$ is easier.


We now consider the restriction of the tiles in $\overline \H_2^\bullet$ to the $4 \times j$-board (first $j$ columns, where $j \geq 2$), for which column $j$ is a fixed configuration.
Below, we count these partial tiles.
The following $6$ quantities count cases, where all contained/occupied cells are connected with the first row (within the partial tile). 
We call such a cell a \emph{fixed cell}; a cell that is occupied and connected with the first column by means of a path going to the left. Let:
\begin{itemize}
    \item $x(j)$ be the number of tiles for which column $j$ contains one fixed particular cell (i.e. a cell occupied and connected)
    \item $\overline x(j)$ be the number of tiles for which column $j$ contains one particular cell among the first $3$, and the final cell
    \item $y(j)$ be the number of tiles for which column $j$ contains two fixed cells among the first three
    \item $\overline y(j)$ be the number of tiles for which column $j$ contains two particular cells among the first $3$, and the final cell
    \item $z(j)$ be the number of tiles for which column $j$ contains the first three cells
    \item $\overline z(j)$ be the number of tiles for which column $j$ contains all cells 
\end{itemize}

Note that some of these cases appear multiple times. E.g. there are $3$ different tiles for the one in column $j$, where for each of them, the number of partial tiles ending with that column are counted by $x(j)$.

We also consider the number of tiles for which some cells are not yet connected within the partial tile, which we call \emph{loose cells}.
A loose cell in column $j$ is an occupied cell which is not connected within the partial $4 \times j$-board (or not connected with a path going to the left towards the first column). Let:

\begin{itemize}
    \item $y_1(j)$ be the number of tiles for which column $j$ contains a fixed cell and a fixed loose cell (among the first three)
    \item $z_1(j)$ be the number of tiles for which column $j$ contains two fixed cells and $1$ fixed loose cell (among the first three)
    \item $z_2(j)$ be the number of tiles for which column $j$ contains one fixed cell and $2$ fixed loose cells (among the first three)
\end{itemize}

Note that $x(2)=\overline x(2)=y(2)=\overline y(2)=z(2)=\overline z(2)=1$ and $y_1(2)=z_1(2)=z_2(2)=0.$

One can now recursively count each of these quantities.

\begin{align*}
    x(j+1)&=x(j)+\overline x(j) + 2y(j)+2\overline y(j) +z(j)+\overline z(j) \\
    \overline x(j+1)&=x(j)+\overline x(j) + 2y(j)+2\overline y(j) +z(j)+\overline z(j) \\
    y(j+1)&=y(j)+\overline y(j) +z(j)+\overline z(j) \\
    \overline y(j+1)&=2x(j)+2\overline x(j)+3y(j)+3\overline y(j) +z(j)+\overline z(j)+2y_1(j)+2z_1(j) \\
    z(j+1)&=z(j)+\overline z(j) \\
    \overline z(j+1)&=3x(j)+3\overline x(j)+3y(j)+3\overline y(j)+z(j)+\overline z(j)+6y_1(j)+3z_1(j)+3z_2(j)\\
    y_1(j+1)&=x(j)+\overline x(j) +y(j)+\overline y(j)+y_1(j)+z_1(j)\\
    z_1(j+1)&=y(j)+\overline y(j)+z_1(j)\\
    z_2(j+1)&=x(j)+\overline x(j) +2y_1(j)+z_2(j)
\end{align*}

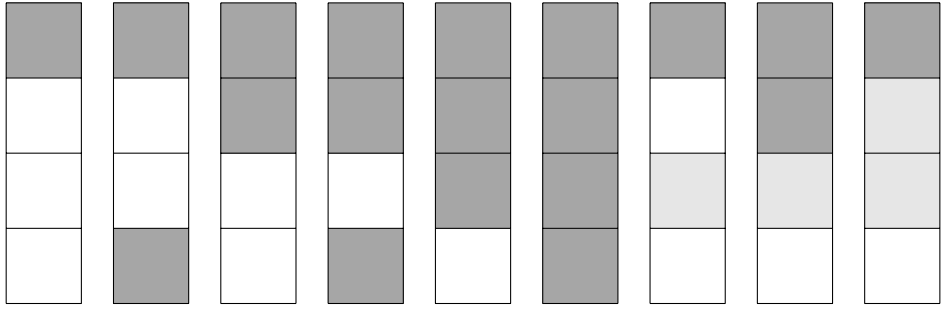
\begin{figure}
    \centering
    \begin{tikzpicture}
        \fill[gray!70!white] (0,3) rectangle (1,4);
        \draw[step=1cm] (0,0) grid (1,4);
    \end{tikzpicture}\quad
    \begin{tikzpicture}
        \fill[gray!70!white] (0,3) rectangle (1,4);
         \fill[gray!70!white] (0,0) rectangle (1,1);
        \draw[step=1cm] (0,0) grid (1,4);
    \end{tikzpicture}\quad
    \begin{tikzpicture}
        \fill[gray!70!white] (0,2) rectangle (1,4);
        \draw[step=1cm] (0,0) grid (1,4);
    \end{tikzpicture}\quad
    \begin{tikzpicture}
        \fill[gray!70!white] (0,2) rectangle (1,4);
        \fill[gray!70!white] (0,0) rectangle (1,1);
        \draw[step=1cm] (0,0) grid (1,4);
    \end{tikzpicture}\quad
    \begin{tikzpicture}
        \fill[gray!70!white] (0,1) rectangle (1,4);
        \draw[step=1cm] (0,0) grid (1,4);
    \end{tikzpicture}\quad
    \begin{tikzpicture}
        \fill[gray!70!white] (0,0) rectangle (1,4);
        \draw[step=1cm] (0,0) grid (1,4);
    \end{tikzpicture}\quad
    \begin{tikzpicture}
        \fill[gray!70!white] (0,3) rectangle (1,4);
        \fill[gray!20!white] (0,1) rectangle (1,2);
        \draw[step=1cm] (0,0) grid (1,4);
    \end{tikzpicture}\quad
    \begin{tikzpicture}
        \fill[gray!70!white] (0,2) rectangle (1,4);
        \fill[gray!20!white] (0,1) rectangle (1,2);
        \draw[step=1cm] (0,0) grid (1,4);
    \end{tikzpicture}\quad
    \begin{tikzpicture}
        \fill[gray!70!white] (0,3) rectangle (1,4);
        \fill[gray!20!white] (0,1) rectangle (1,3);
        \draw[step=1cm] (0,0) grid (1,4);
    \end{tikzpicture}
    \caption{Examples of columns counted by $x,\overline x, y, \overline y, z ,\overline z, y_1, z_1$ and $z_2$ respectively}
    \label{fig:enter-label}
\end{figure}

Consider the coefficient matrix
$$A=\begin{bmatrix}
    1 & 1 & 2 & 2 & 1 & 1 & 0 & 0 & 0 \\
    1 & 1 & 2 & 2 & 1 & 1 & 0 & 0 & 0 \\
    0 & 0 & 1 & 1 & 1 & 1 & 0 & 0 & 0 \\
    2 & 2 & 3 & 3 & 1 & 1 & 2 & 2 & 0 \\
    0 & 0 & 0 & 0 & 1 & 1 & 0 & 0 & 0 \\
    3 & 3 & 3 & 3 & 1 & 1 & 6 & 3 & 3 \\
    1 & 1 & 1 & 1 & 0 & 0 & 1 & 1 & 0 \\
    0 & 0 & 1 & 1 & 0 & 0 & 0 & 1 & 0 \\
    1 & 1 & 0 & 0 & 0 & 0 & 2 & 0 & 1 \\
\end{bmatrix}$$

Then $\abs{\overline \H_2^\bullet}=\left(A^{n}\vec{e_6}\right)_6$.
Let $\lambda$ be the largest eigenvalue of $A$.
The desired value $c(G)$ can now be computed as $ \sqrt[4]{\lambda}.$
For this graph, it equals approximately $ \sqrt[4]{8.95242} > 1.729$.
Note here that the number is clearly at least $2^{k-1}$ and the second largest absolute value of an eigenvalue of $A$ was smaller than $2.$

\section{Lower bounds on $c_d$ and $\ct_d$ for small $d$}\label{sec:lowerbounds}

We can extend the construction of~\cref{sec:cubic_fam1} towards families yielding better lower bounds.
Let $M$ be a $d$-regular graph (of small order $n_0$) and $C=\{v_1, v_2, \ldots, v_{\ell}\}$ a cycle in $M$, of length $\ell$ (we will refer to this $M$ as a \textit{gadget}). Here $V(M)=\{v_1, v_2, \ldots, v_{n_0} \}.$ In~\cref{sec:cubic_fam1}, $M=K_4$ and $C=C_3.$
We will choose $M$ to be a Moore graph later and $C$ a cycle for which all vertices have the same degree in $M[V(C)]$ (even while this is not necessary for the general construction). Note that the choice of the cycle $C$ is not always unique.

Let $G$ be the graph on $k n_0$ vertices, with vertices $v_{i,j},$ $i \in [n_0], j \in [k].$
The adjacencies in $G$ are defined as follows:
\begin{itemize}
    \item $v_{i,j}$ and $v_{i', j}$ are adjacent whenever $v_iv_{i'} \in E(M) \backslash E(C),$
    \item $v_{i,j}$ and $v_{i,j+1}$ are adjacent for all $i \in [\ell]$ and $j \in [k-1]$,
    \item $v_{i, 1}$ and $v_{i-1,k}$ are adjacent for $i \in [\ell]$ where $v_{0,k}$ denotes $v_{\ell,k}.$
\end{itemize}

One can consider this as gluing $k$ copies of $M$ together, along the vertices on the cycle $C$. We will denote the final graph $G$ as $M^{\bullet k}.$ In~\cref{fig:bullet2}, this is illustrated for $M \in \{P_{5,2}, K_{4,4} \}$ and $k=2,$ where the cycle is a $C_6$ in both cases.

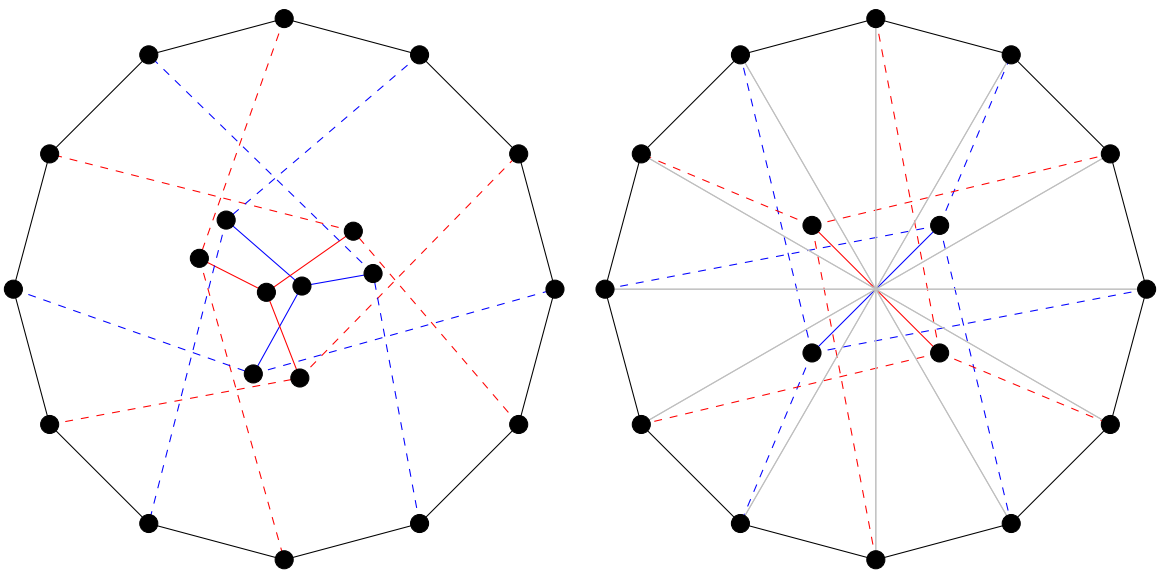
\begin{figure}[h]
\centering
\begin{tikzpicture}[scale=0.8]

\foreach \x in {0,2,4}{
\draw[blue] (\x*60+10:1.5) --(10:0.3);
}
\foreach \x in {1,3,5}{
\draw[red] (\x*60-20:1.5) --(190:0.3);
}

\foreach \x in {1,3,5}{
\draw[red, dashed] (60*\x-90:4.5)--(\x*60-20:1.5) --(60*\x+90:4.5);
}

\foreach \x in {0,2,4}{
\draw[blue, dashed] (60*\x-60:4.5)--(\x*60+10:1.5) --(60*\x+120:4.5);
}

\foreach \x in {0,1,...,11}{
\draw[fill] (\x*30:4.5) circle (0.15);
\draw (\x*30+30:4.5)--(\x*30:4.5);
}

\draw[fill] (10:0.3) circle (0.15);
\draw[fill] (190:0.3) circle (0.15);
\foreach \x in {0,2,4}{\draw[fill] (\x*60+10:1.5) circle (0.15);
}

\foreach \x in {1,3,5}{
\draw[fill] (\x*60-20:1.5) circle (0.15);
}

\end{tikzpicture}\quad
\begin{tikzpicture}[scale=0.8]

\draw[blue] (45:1.5) -- (225:1.5);
\draw[red] (90+45:1.5) -- (90+225:1.5);

\foreach \x in {0,2}{
\foreach \y in {0,1,2}{
\draw[blue, dashed] (\x*90+45:1.5) -- (\y*120+60+30*\x:4.5) ;
}
}

\foreach \x in {1,3}{
\foreach \y in {0,1,2}{
\draw[red, dashed] (\x*90+45:1.5) -- (\y*120+30*\x:4.5) ;
}
}

\foreach \x in {0,1,2,3}{\draw[fill] (\x*90+45:1.5) circle (0.15);
}

\foreach \x in {0,1,...,11}{
\draw[lightgray] (\x*30+180:4.5)--(\x*30:4.5);
}

\foreach \x in {0,1,...,11}{
\draw[fill] (\x*30:4.5) circle (0.15);
\draw (\x*30+30:4.5)--(\x*30:4.5);
}

\end{tikzpicture}
\caption{$P_{5,2}^{\bullet 2}$ and $K_{4,4}^{\bullet 2}$}\label{fig:bullet2}
\end{figure}

For $d=3,$ we consider $4$ families by choosing $M$ to be $K_4, K_{3,3},$ the Petersen graph $P_{5,2}$ and the Heawood graph $H_{3,6}.$

Similarly to Section~\ref{sec:cubic_fam1}, for the other $3$ families the eigenvalues of an associated coefficient matrix can be computed, and by taking the respective $6^{th}, 10^{th}$ or $14^{th}$ root of the largest eigenvalue (checking that the second largest eigenvalue cannot be dominant), we derive $c(G)$ for the families.
For $d \in \{4,5\},$ we focus on the family obtained from $M=K_{d,d}.$ 

As we will consider $N_{Dom}(G)$ as well, we first explain the determination of $\ct(G)$ for the family considered in~\cref{sec:cubic_fam1}, which is precisely $K_4^{\bullet k}.$

\subsection{Lower bound for the number of dominating connected sets in $K_4^{\bullet k}$}

In this subsection, we determine a lower bound for $\ct_3$ by estimating (i.e. determining up to a polynomial factor) $N_{dom}(G)$ for $G=K_4^{\bullet k}$ where $k \to \infty.$ The vertices can again be represented as cells in a $4 \times k$-board.
We have to consider four types of cells.
Cells can be occupied (both fixed and loose) as well as unoccupied (dominated or not dominated by cells in its column or the one left to it). Having four types of cells, there are more cases for the possible form of the column than was the case in Section~\ref{sec:cubic_fam1}.
It turns out there are $13$ of them up to isomorphism, and these are represented in~\cref{fig:13tiles}. 
Here permuting the first $3$ cells in a column is considered as an isomorphic column (but the counting is done for a fixed column).

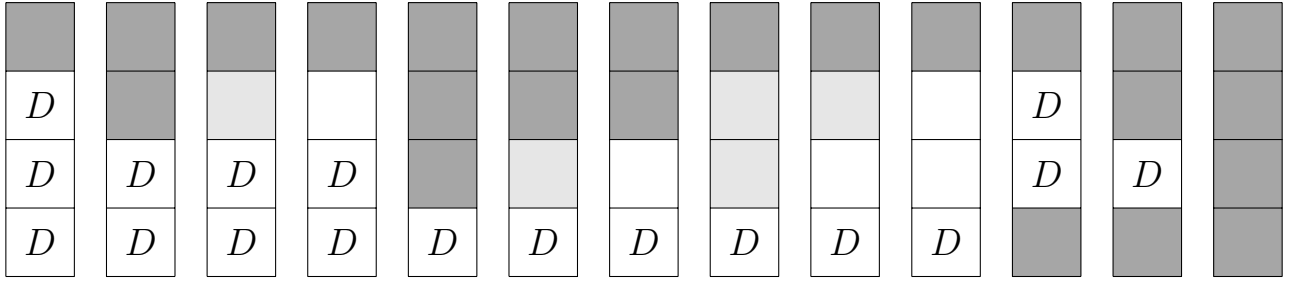
\begin{figure}[h]
    \centering
    \begin{tikzpicture}[scale=0.911]
                 \node at (0.5,0.5) {\large $D$};
          \node at (0.5,1.5) {\large $D$};
                   \node at (0.5,2.5) {\large $D$};
        \fill[gray!70!white] (0,3) rectangle (1,4);
        \draw[step=1cm] (0,0) grid (1,4);
    \end{tikzpicture}\quad
    \begin{tikzpicture}[scale=0.911]
        \fill[gray!70!white] (0,2) rectangle (1,4);
        \draw[step=1cm] (0,0) grid (1,4);
        \node at (0.5,0.5) {\large $D$};
        \node at (0.5,1.5) {\large $D$};
    \end{tikzpicture}\quad
    \begin{tikzpicture}[scale=0.911]
        \fill[gray!70!white] (0,3) rectangle (1,4);
        \fill[gray!20!white] (0,3) rectangle (1,2);
          \node at (0.5,0.5) {\large $D$};
          \node at (0.5,1.5) {\large $D$};
        \draw[step=1cm] (0,0) grid (1,4);
    \end{tikzpicture}\quad
    \begin{tikzpicture}[scale=0.911]
        \fill[gray!70!white] (0,3) rectangle (1,4);
        \draw[step=1cm] (0,0) grid (1,4);
         \node at (0.5,0.5) {\large $D$};
          \node at (0.5,1.5) {\large $D$};
    \end{tikzpicture}\quad
    \begin{tikzpicture}[scale=0.911]
        \fill[gray!70!white] (0,1) rectangle (1,4);
        \node at (0.5,0.5) {\large $D$};
        \draw[step=1cm] (0,0) grid (1,4);
    \end{tikzpicture}\quad
   \begin{tikzpicture}[scale=0.911]
        \fill[gray!70!white] (0,2) rectangle (1,4);
        \fill[gray!20!white] (0,1) rectangle (1,2);
          \node at (0.5,0.5) {\large $D$};
        \draw[step=1cm] (0,0) grid (1,4);
    \end{tikzpicture}\quad
    \begin{tikzpicture}[scale=0.911]
        \fill[gray!70!white] (0,2) rectangle (1,4);
        \draw[step=1cm] (0,0) grid (1,4);
        \node at (0.5,0.5) {\large $D$};
    \end{tikzpicture}\quad
    \begin{tikzpicture}[scale=0.911]
        \fill[gray!70!white] (0,3) rectangle (1,4);
        \fill[gray!20!white] (0,1) rectangle (1,3);
          \node at (0.5,0.5) {\large $D$};
        \draw[step=1cm] (0,0) grid (1,4);
    \end{tikzpicture}\quad
    \begin{tikzpicture}[scale=0.911]
        \fill[gray!70!white] (0,3) rectangle (1,4);
        \fill[gray!20!white] (0,3) rectangle (1,2);
          \node at (0.5,0.5) {\large $D$};
        \draw[step=1cm] (0,0) grid (1,4);
    \end{tikzpicture}\quad
    \begin{tikzpicture}[scale=0.911]
          \node at (0.5,0.5) {\large $D$};
        \fill[gray!70!white] (0,3) rectangle (1,4);
        \draw[step=1cm] (0,0) grid (1,4);
    \end{tikzpicture}\quad
     \begin{tikzpicture}[scale=0.911]
        \fill[gray!70!white] (0,3) rectangle (1,4);
         \fill[gray!70!white] (0,0) rectangle (1,1);
                  \node at (0.5,1.5) {\large $D$};
                   \node at (0.5,2.5) {\large $D$};
        \draw[step=1cm] (0,0) grid (1,4);
    \end{tikzpicture}\quad
    \begin{tikzpicture}[scale=0.911]
        \fill[gray!70!white] (0,2) rectangle (1,4);
        \fill[gray!70!white] (0,0) rectangle (1,1);
                 \node at (0.5,1.5) {\large $D$};
        \draw[step=1cm] (0,0) grid (1,4);
    \end{tikzpicture}\quad
    \begin{tikzpicture}[scale=0.911]
        \fill[gray!70!white] (0,0) rectangle (1,4);
        \draw[step=1cm] (0,0) grid (1,4);
    \end{tikzpicture}
    \caption{The $13$ types of columns necessary to estimate $N_{dom}(K_4^{\bullet k})$}
    \label{fig:13tiles}
\end{figure}

A tile corresponding to a connected dominating set, will contain at least one cell in every column.
Let $\H$ be the set of all such tiles.
Denote with $\overline \H^\bullet$ the (dominating) tiles which contain all cells of the first column and for which every other occupied cell is connected to the first column with a path going to the first column in the left direction.
Then $\abs \H \le 225 n \abs { \overline \H^\bullet} $ and as such it is sufficient to know the exponential growth of $\abs{\overline \H^\bullet}.$
We can do so by counting the number of such tiles restricted to the first $j$ columns and count them with a recursion, determined by a coefficient matrix.

The coefficient matrix is the following.
\setcounter{MaxMatrixCols}{20}
$$
A=
\begin{bmatrix}
    0 & 0 & 0 & 0 & 1 & 0 & 0 & 0 & 0 & 0 & 0 & 0 & 1 \\
    0 & 0 & 0 & 0 & 1 & 0 & 0 & 0 & 0 & 0 & 0 & 0 & 1 \\
    0 & 1 & 0 & 0 & 0 & 1 & 1 & 0 & 0 & 0 & 0 & 1 & 0 \\
    0 & 1 & 0 & 0 & 0 & 0 & 0 & 0 & 0 & 0 & 0 & 1 & 0 \\
    0 & 0 & 0 & 0 & 1 & 0 & 0 & 0 & 0 & 0 & 0 & 0 & 1 \\
    0 & 1 & 0 & 0 & 0 & 1 & 1 & 0 & 0 & 0 & 0 & 1 & 0 \\
    0 & 1 & 0 & 0 & 0 & 0 & 0 & 0 & 0 & 0 & 0 & 1 & 0 \\
    1 & 0 & 2 & 2 & 0 & 0 & 0 & 1 & 2 & 1 & 1 & 0 & 0 \\
    1 & 0 & 1 & 1 & 0 & 0 & 0 & 0 & 0 & 0 & 1 & 0 & 0 \\
    1 & 0 & 0 & 0 & 0 & 0 & 0 & 0 & 0 & 0 & 1 & 0 & 0 \\
    1 & 2 & 0 & 0 & 1 & 0 & 0 & 0 & 0 & 0 & 1 & 2 & 1 \\
    2 & 3 & 2 & 2 & 1 & 2 & 2 & 0 & 0 & 0 & 2 & 3 & 1 \\
    3 & 3 & 6 & 6 & 1 & 3 & 3 & 3 & 6 & 3 & 3 & 3 & 1 \\
\end{bmatrix}
$$
The largest eigenvalue of the coefficient matrix $A$ is now $\lambda \sim 8.29488091$.
All other eigenvalues satisfy $\abs{\lambda_i} < 2.31$, while there are clearly more than $4^k$ different dominating connected sets in $K_4^{\bullet k}.$
For the latter, observe that an arc with $k$ consecutive vertices, together with $k$ internal vertices, span a connected dominating set.
This implies that the main term is $\Theta( \lambda^k)$ and thus $\ct_3 \ge \sqrt[4]{\lambda}>1.697$.

\subsection{Better lower bounds for $c_d, \ct_d$ for $d \in \{3,4,5\}$}

We can compute a lower bound for $c_d$ and $\ct_d$ by considering the general construction $M^{\bullet k}.$
Similarly as for $K_4^{\bullet k},$ one can assign a $\abs{M} \cdot k$-board and consider the number of intensely connected tiles with the additional property that there is always at least one cell in the $j^{th}$ column connected to the first within the partial $\abs{M} \cdot j$-board.
Using a computer program, see~\cref{sec:app1}, the associated coefficient matrix $A$ is computed for $K_{d+1}^{\bullet k}, K_{d,d}^{\bullet k}$ , $ P_{5,2}^{\bullet k}$ and $H_{3,6}^{\bullet k}.$

The important cases have been summarized in~\cref{tab:overview_lb_rotatingM}.
Here $\lambda$ is the largest eigenvalue of $A$ (which is a positive real), and $\abs{\lambda}_2$ is the second largest modulus/ absolute value of all eigenvalues.

\begin{table}[h]
    \centering
    \begin{tabular}{|ccc| c c c| c c c | }
        \hline
         & & & \multicolumn{3}{|c|}{connected} & \multicolumn{3}{|c|}{dominating conn.} \\
        $d$& $M$ & $C$ & $\lambda$ & $\abs{\lambda}_2$ & $c_d \ge$ & $\lambda$ & $\abs{\lambda}_2$ & $\ct_d \ge$\\
        \hline
        3&$K_{3,3}$ & $C_4$ & 30.30  & 7.81 & 1.766 &  26.15 & 6.79 & 1.723\\
        3&$P_{5,2}$ & $C_6$ & 329.81 & 126.17 & 1.786 &  252.7 &  100.47 & 1.739 \\
        3&$H_{6,3}$ & $C_8$ & 3512.31 & 1145.53 & 1.792 &  2396.46 & 
858.38 & 1.743\\
        4&$K_{4,4}$ & $C_6$ & 167.97 & 29.86 & 1.897 &152.66&
31.33 &1.875\\
        5&$K_{5,5}$ & $C_8$ & 807.93 & 90.47 & 1.953 & 756.64&
102.54& 1.940\\
        \hline
    \end{tabular}
    \caption{Largest two eigenvalues for the matrix related to the family $M^{\bullet k}$}
    \label{tab:overview_lb_rotatingM}
\end{table}

We can construct a connected dominating set by taking a $k$-arc on the outer-cycle, together with the vertices in the inner part.
That is, the set $v_{i, j}$ for $\ell \le i \le n_0$ and $j \in [k]$, is connected and dominating.
Since we can extend this one arbitrarily, there are at least $2^{(\ell-1)k}$ many connected dominating sets.
When $G=K_{d,d}^{\bullet k},$ we have that $\ell=2(d-1)$ and we note that $\abs{\lambda}_2<2^{2d-1}$ for both the connected and dominating connected case.
This proves that the number of tiles in both cases indeed grows as $\Theta( \lambda^k).$
We stress here that the derived lower bounds on $c_d$ and $\ct_d$ for $M=K_{d,d}$ beat the best known lower bounds by~\cite{KKK18}, despite the much smaller base graph. This implies that the new construction is somewhat more efficient.

For $G=P_{5,2}^{\bullet k}$ and $G=H_{6,3}^{\bullet k},$ we cannot conclude as before.
A connected dominating set of order $k$, can dominate at most $k+2$ vertices, i.e. a connected dominating set needs to have at least $\frac n2 -1$ vertices. See e.g.~\cite{Duck02} for estimates for random cubic graphs on the size of a minimum connected dominating set.
As such, by considering all extensions of one possible connected dominating set, we cannot have more than $2^{n/2+1}$ of them and conclude.
We now end with a different proof that the lower bound for $c_d$ and $\ct_d$ will always equal $\sqrt[\abs M]{\lambda}.$

\begin{prop}
    For a fixed graph $M$, as $k$ tends to infinity, we have that for $G=M^{\bullet k}$, $N(G)=\Theta(\lambda^k)$ and $N_{dom}(G)=\Theta(\lambda^k)$, where $\lambda$ is the maximum eigenvalue of the respective coefficient matrix $A$.
\end{prop}

\begin{proof}
    We start with a simple claim, telling us that every possible state (of a column) can appear anywhere (and that there is a nowhere-zero row in $A$).
    \begin{claim}\label{clm:fullcolumn_nearby}
        Let $c$ be any possible column.
        Then $c$ can be succeeded by the column with all cells occupied.
        Also, the second predecessor (two columns to the left) can be the column with all cells occupied.
    \end{claim}
    \begin{claimproof}
        The first part is trivial. An occupied (both loose and fixed) cell, as well as a dominating or non-dominating non-occupied cell, can be followed by a fixed occupied cell.
        For the second part, note that every fixed occupied and dominated non-occupied cell can be preceded by a fixed occupied (FO) cell.
        Loose occupied - and non-dominated non-occupied cells can be preceded by a dominated non-occupied (D) cell.
        Using this twice, we note that the column can be preceded by a column having only FO- and D-cells, and the latter can be preceded by a full FO-column.
    \end{claimproof}
    The eigenvector corresponding with the largest eigenvalue $\lambda$ of $A$ is not the zero vector, and as such there is at least one base state (column) / base vector $\vec v$ which is not orthogonal to it and thus $A^j \vec{v} = \Theta( \lambda ^j).$
    Starting from this column, we can construct $\Theta( \lambda ^{n-3})$ partial boards. Each of them can be extended to a full board / tile by adding a full FO-column and another column, due to~\cref{clm:fullcolumn_nearby}. 
    This implies that there are at least $\Theta( \lambda ^k)$ tiles of the desired form.
\end{proof}

\section{Small extremal cubic and quartic graphs}\label{sec:d34_smallorder}

Using two independent algorithms, we computationally determined $c_{3,g}(n)$, $\ct_{3,g}(n)$, $c_{4,g}(n)$ and $\ct_{4,g}(n)$ for small girths and small orders (see~\cref{tab:c_{3,g}(n)} for the 3-regular case and~\cref{tab:c_{4,g}(n)} in~\cref{sec:app2} for the 4-regular case). More details about the algorithms can be found in~\cref{sec:app1}.
For each case that we computed, the extremal graph was unique except for $\ct_{4,3}(12)$, which had two extremal graphs. All the extremal graphs 
can also be inspected in the database of interesting graphs from the \textit{House of Graphs}~\cite{HoG} by searching for the keywords ``graph with many connected * subgraphs''. The most interesting case is $g=3$. Since the number of 3- and 4-regular graphs grows rapidly as $n$ increases, looking at higher girths enables us to exhaustively investigate larger orders as well. Moreover, for each extremal graph we determined the girth and saw that this tends to increase as the order grows. This indeed agrees with what one could intuitively expect. Hence, it is likely that many of the extremal graphs that we found when restricting the search to higher girths are extremal for lower girths as well (perhaps even for $g=3$). The values of $c_{d,g}(n)$ shown in bold have Moore graphs as the extremal graphs. These computations provide further evidence for~\cref{conj:Mooregraphs_extremal} and give rise to the (conditional) upper bounds for $c_3$ and $c_4$ from~\cref{tab:overview_c3_c4}. 
Another important observation is that $\ct_{3,7}(30)$ and hence $\ct_3(30)$ is not obtained by the (Moore graph) Tutte–Coxeter graph, $H_{3,8}$.
Except for small orders, the extremal graphs for $N$ and $N_{dom}$ seem to be different.

\begin{table}[h]
    \centering
    \begin{tabular}{|c| c | c c | c c |}
\hline
$n$ & $g$ & $c_{3,g}(n)$ & \text{g. ex.} & $\ct_{3,g}(n)$ & \text{g. ex.} \\
\hline
4 & 3 & \textbf{1.9680} & 3 & 1.9680 & 3 \\
6 & 3 & \textbf{1.9501} & 4 & 1.9129 & 4 \\
8 & 3 & 1.9044 & 4 & 1.8358 & 4 \\
10 & 3 & \textbf{1.8855} & 5 & 1.8127 & 5 \\
12 & 3 & 1.8644 & 5 & 1.7957 & 5\\
14 & 3 & \textbf{1.8563} & 6 & 1.7860 & 6\\
16 & 3 & 1.8451 & 6 & 1.7779 & 5\\
18 & 3 & 1.8390 & 6 & 1.7734 & 6\\
20 & 3 & 1.8340 & 6 & 1.7703 & 5\\
22 & 3 & 1.8303 & 6 & 1.7676 & 6\\
24 & 6 & 1.8275 & 7 & 1.7656 & 6\\
26 & 6 & 1.8249 & 7 & 1.7641 & 6\\
28 & 7 & 1.8229 & 7 & 1.7622 & 7\\
30 & 7 & \textbf{1.8214} & 8 & 1.7616 & 7\\
\hline
\end{tabular}
    \caption{Summary of the computations of $c_{3,g}(n)$,  $\ct_{3,g}(n)$ (rounded up to $4$ decimals) and the girths of the extremal graphs (g. ex.)}
    \label{tab:c_{3,g}(n)}
\end{table}

Finally, we focus on the extremal graphs for the main case: $c_3(n)$, which is determined exactly for $n \le 22.$
We note that some of the extremal graphs are members of the infinite families described in~\cref{sec:cubic_fam1} and~\cref{sec:lowerbounds}, e.g. $K_4^{\bullet 4}$ and $K_{3,3}^{\bullet 2}$.
The extremal graphs for $n \le 16$ are the well-known graphs $K_4$, $K_{3,3}$ , the $8$-vertex Möbius ladder, the Petersen graph $GP(5,2)$, the twinplex, the Heawood graph and the Möbius-Kantor graph $GP(8,3)$.

For $18 \le n \le 22$, the graphs attaining $c_3(n)$ are not that famous. The extremal graphs for $n \in \{20, 22\}$ are depicted in Figure~\ref{fig:c_3(20,22)}. 

\begin{figure}[h]
\centering
\begin{tikzpicture}[scale=0.6]
\foreach \x in {0,1,...,13}{
\draw[fill] (\x*360/14:4.5) circle (0.15);
\draw (\x*360/14+360/14:4.5)--(\x*360/14:4.5);
}
\draw[fill] (1.25,0) circle (0.15);
\draw[fill] (-1,0) circle (0.15);
\draw (1.25,0)--(-1,0);
\draw[fill] (1.5,-1.5) circle (0.15);
\draw[fill] (-1.5,-1) circle (0.15);
\draw (1.25,0)--(-1,0);

\draw (1.25,0)--(-1,0);
\draw (0*360/14:4.5)--(1.25,0)--(6*360/14:4.5);
\draw (3*360/14:4.5)--(-1,0)--(-4*360/14:4.5);

\draw[blue] (1.5,-1.5)--(-1.5,-1);
\draw[blue] (1*360/14:4.5)--(1.5,-1.5)--(-6*360/14:4.5);
\draw[blue] (5*360/14:4.5)--(-1.5,-1)--(-3*360/14:4.5);

\draw[fill, orange] (2,-2) circle (0.15);
\foreach \x in {-1,4,9}{
\draw[fill, orange] (\x*360/14:4.5) circle (0.15);
\draw[orange, dashed] (2,-2)--(\x*360/14:4.5);
}

\draw[fill, red] (2,2) circle (0.15);
\foreach \x in {-2,2,7}{
\draw[fill, red] (\x*360/14:4.5) circle (0.15);
\draw[red, dashed] (2,2)--(\x*360/14:4.5);
}

\end{tikzpicture}
\quad
\begin{tikzpicture}[scale=0.6]

\draw[blue, dashed] (210:4.5)--(100:3)--(30:4.5);
\draw[blue, dashed] (300:4.5)--(160:3)--(120:4.5);

\draw[red, dashed] (180:4.5)--(220:3)--(330:4.5);
\draw[red, dashed] (240:4.5)--(280:3)--(90:4.5);

\draw[orange, dashed] (0:4.5)--(-20:3)--(150:4.5);
\draw[gray, dashed] (60:4.5)--(40:3)--(-90:4.5);

\foreach \x in {0,1,...,12}{
\draw[fill] (\x*360/12:4.5) circle (0.15);
\draw (\x*360/12+360/12:4.5)--(\x*360/12:4.5);
}
\draw[fill] (0,0) circle (0.15);
\foreach \x in {0,1,2}{
\draw[fill] (\x*120+10:1.5) circle (0.15);
\draw[fill] (\x*120+40:3) circle (0.15);
\draw[fill] (\x*120-20:3) circle (0.15);
\draw (0,0)--(\x*120+10:1.5);
\draw (120*\x+40:3)--(\x*120+10:1.5);
\draw (120*\x-20:3)--(\x*120+10:1.5);
}
\foreach \x in {0,1,...,12}{
\draw (\x*360/12+360/12:4.5)--(\x*360/12:4.5);
}
\foreach \x in {0,1,...,12}{
\draw[fill, gray] (\x*360/12:4.5) circle (0.15);
}
\foreach \x in {3,6,8,11}{
\draw[fill, red] (\x*360/12:4.5) circle (0.15);
}
\foreach \x in {1,4,7,10}{
\draw[fill, blue] (\x*360/12:4.5) circle (0.15);
}
\end{tikzpicture}
\caption{The extremal graphs for $c_3(20)$ and $c_3(22)$}\label{fig:c_3(20,22)}
\end{figure}
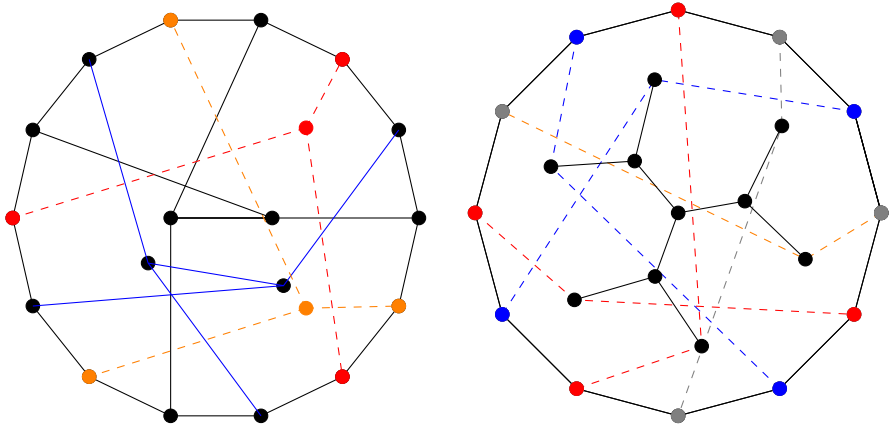

These graphs can be drawn in a nice way except for a few edges that break the apparent symmetry, indicating that a general pattern may not be possible to find.
Still there are surprising occurrences of well-known graphs;
among the graphs with girth at least $6$ and order $24$ and $26$, the McGee graph and $GP(13,5)$ turn out to be extremal. In~\cref{fig:D_{16,26}graph}, the generalized Petersen graphs $GP(8,3)$ and $GP(13,5)$ are presented.
Note that we expect that the extremal graph attaining $c_3(n)$ has large girth for large $n$ and the girth of a generalized Petersen graph is bounded by $8$.
It seems hard to predict the exact extremal graph for $c_3(n)$ for any large $n$, except for $n=126$, where we expect the Moore graph $H_{3,12}$ to be extremal. Cages (the smallest $d$-regular graphs with girth $g$) are natural candidates, but there are arguments not to believe all of them are extremal for $c_3(n)$. 
For example, as there are multiple cages for certain combinations of $(d,g)$, e.g. $(3, 10)$-cages.

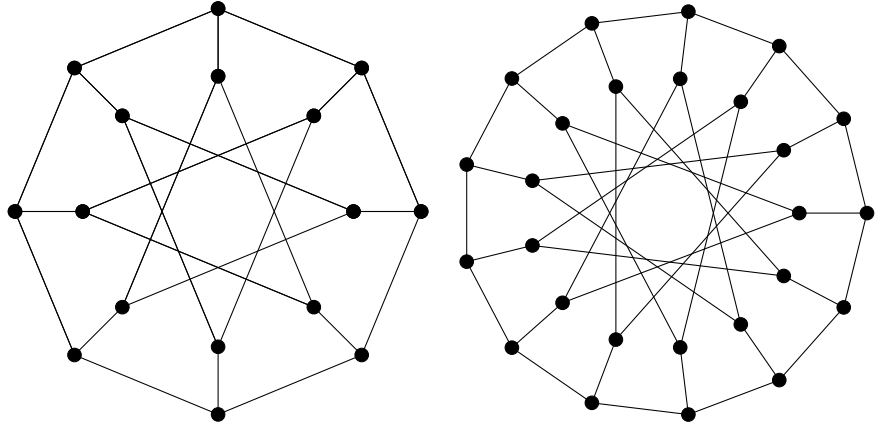
\begin{figure}[ht]
\centering
\begin{tikzpicture}[scale=0.6]
\foreach \x in {0,1,...,12}{
\draw[fill] (\x*360/8:4.5) circle (0.15);
\draw[fill] (\x*360/8:3) circle (0.15);
\draw (\x*360/8:3)--(\x*360/8:4.5);
\draw (\x*360/8+360/8:4.5)--(\x*360/8:4.5);
\draw (\x*360/8+1080/8:3)--(\x*360/8:3);
}
\end{tikzpicture}\quad
\begin{tikzpicture}[scale=0.6]
\foreach \x in {0,1,...,12}{
\draw[fill] (\x*360/13:4.5) circle (0.15);
\draw[fill] (\x*360/13:3) circle (0.15);
\draw (\x*360/13:3)--(\x*360/13:4.5);
\draw (\x*360/13+360/13:4.5)--(\x*360/13:4.5);
\draw (\x*360/13+1800/13:3)--(\x*360/13:3);
}
\end{tikzpicture}
\caption{The extremal graphs for $c_3(16)$ and $c_{3,6}(26)$}\label{fig:D_{16,26}graph}
\end{figure}

\section{Conclusion}\label{sec:conc}

In this paper, we studied the fundamental notion of the number of connected sets in a graph and the maximum, by means of investigating the behavior of $c_d(n)$.
From our results, we immediately conclude that $c_2(n)$ is monotone decreasing for (integral) $n \ge 3$, while this seems not to be true for $c_d(n)$ whenever $d>3$ by~\cref{prop:cd_notmono_d4} and~\cref{prop:nonmonotone}. This leaves the case $d=3$. Here the comparison of $c_3(124)$ and $c_3(126)$ may show non-monotonicity (the Tutte 12-cage has order 126).
Haslegrave~\cite{John22b} asked about the convergence of $c_d$ going to $2$, observing that $c_d> 2^{1-2/(d-1)}$.
Given expansion properties and the intuition that most non-connected vertex subsets have a small component, could it be possible that $c_d=2-\Theta(2^{-d})?$

We determined the extremal graphs for small orders with a computer search.
These extremal graphs and the conjectured Moore graphs are expander graphs (in particular $d$-connected), so as suggested by Vince~\cite{Vince20}, expansion plays a role.

For a lower bound on the asymptotic behavior, we constructed new families of graphs, where essentially copies of a small graph are glued together in a cyclic way. Calling the small graph our gadget, we note that our construction for the gadget $K_{d,d}$ already beats the previous best known lower bounds in~\cite{KKK18}. In the latter case they have aligned the gadgets in series, forming generalized ladder graphs.
To estimate $N(G)$, we computed the exponential behavior by finding the eigenvalues of a coefficient matrix of a related combinatorial counting problem.

In~\cite{Bjorklund12}, the authors refer to studying the number of connected sets $\mathcal C$, the number of dominating sets $\mathcal D$ and the number of dominating connected sets $\mathcal{C} \cap \mathcal{D}$.
A quick search shows that the behavior of the cubic graphs maximizing $\abs{ \mathcal C}$ and $\abs{ \mathcal D}$ are very different. 
In particular, the extremal graphs are connected in the first case (cf.~\cref{prop:extremal_is_connected}) and disconnected for the dominating case.
More precisely, if $d+1 \mid n,$ then it is an immediate corollary of the projection method / Shearer's lemma from~\cite{CGFS86} that the union of copies of $K_{d+1}$ maximizes the number of dominating sets, since every closed neighborhood cannot be empty, i.e. $\abs{ \mathcal D } \le \left(2^{d+1}-1\right)^{n/(d+1)}$.
It is thus also not surprising that the graphs maximizing $\abs{ \mathcal C \cap \mathcal D }$ behave differently. In particular, we observed that the cubic Moore graph of girth $8$, $H_{3,8}$, does not attain $\ct_3(30).$

Vince~\cite{Vince20} considered the number of connected sets in a graph as a measure for connectivity. As such, it is natural to compare it with the algebraic connectivity $\mu$ (the second smallest eigenvalue of the Laplacian matrix of the graph).
We remark that the graphs that maximize the algebraic connectivity and the graph that maximizes the number of connected sets, are not equal in general.
E.g. for $d=4$ and $n=7$, the complement of $C_7$ maximizes $\mu$, while the complement of $C_3 \cup C_4$ maximizes $N$ and $N_{dom}.$
For $d=3$ and $n=12$, the same graph maximizes $\mu$ and $N_{dom}$, but is different from the one maximizing $N$.
For $d=3$ and $n=16$, the same graph maximizes $\mu$ and $N$, but is different from the one maximizing $N_{dom}$.
For $d=4$ and $n=9$, there are four graphs maximizing $\mu.$ It is also clear that the number of connected sets and the algebraic connectivity are very different from a computational point a view, where the former seems much harder in general than the latter. We refer the interested reader to~\cite{Kolokolnikov15} bounds on algebraic connectivity of graphs subject to various constraints and to~\cite{deAbreu07,EKJS24} for overviews of algebraic connectivity.

We briefly elaborate on transient sets. A transient set $T \subset V(G)$ for a reference vertex $s$ and an end vertex $u$ has the property that for every vertex $v \not\in N(s) \cap N(u)$:
\begin{itemize}
    \item if $v \in T,$ then $v$ has at least two neigbhours in $T$,
    \item if $v \not \in T,$ then at least two neighbors of $v$ are not in $T$ either.
\end{itemize}
This notion is inspired by the travelling salesman problem, since the first vertices of a walk visiting all vertices (starting from $s$) needs to be a transient set.
While transient sets can be considered as less elementary, the number of transient sets gives a better upper bound for the time complexity of the variant of the dynamic programming solution by~\cite{Bellman62,HK62}, as observed in~\cite{Bjorklund12}.  
Studying the number of transient sets for small graphs has the additional complication that a reference vertex is taken and so one needs to be consistent in comparing, e.g. taking the average over the different vertices. As such, we left out the study for this variant.

This project also leaves some mathematical challenges, with~\cref{conj:maxgiven_nd} and~\cref{conj:Mooregraphs_extremal} as main conjectures with nice implications.

\section*{Acknowledgement}

Stijn Cambie and Jorik Jooken are supported by FWO grants with grant numbers 1225224N and 1222524N, respectively. The research of Jan Goedgebeur was supported by Internal Funds of KU Leuven and an FWO grant with number G0AGX24N. 

The computational resources and services used in this work were provided by the VSC (Flemish Supercomputer Center), funded by the Research Foundation - Flanders (FWO) and the Flemish Government – department EWI.

\paragraph{Open access statement.} For the purpose of open access,
a CC BY public copyright license is applied
to any Author Accepted Manuscript (AAM)
arising from this submission.

\bibliographystyle{abbrv}
\bibliography{ref}

\section*{Appendix}\label{sec: appendix}

\appendix

\section{Details of computer programs}\label{sec:app1}
\subsection*{Determining $c_{d,g}(n)$ and $\ct_{d,g}(n)$}
We used the graph generator GENREG~\cite{Me99} for generating all connected $d$-regular graphs on $n$ vertices with girth at least $g$. We implemented two independent algorithms for calculating $N(G)$, two independent algorithms for $N_{dom}(G)$ and used a computer cluster to do these computations for the generated graphs with all algorithms. These computations amounted to a total time of around 1 CPU-year. The results of all algorithms were compared and were in complete agreement with each other. The first algorithm for calculating $N(G)$ keeps track of a counter and increments this counter by iterating over all subsets $V' \subseteq V(G)$ and using a flood fill algorithm for checking whether the graph induced by $V'$ is connected. The first algorithm for calculating $N_{dom}(G)$ additionally checks for each vertex in $V(G) \setminus V'$ whether it has a neighbor in $V'$. The second algorithm for calculating $N(G)$ ($N_{dom}(G)$) recursively generates all connected (dominating) sets (the pseudocode is given in~\cref{algo:recursivelyCount}). The recursion starts from the set $V(G)$, which is connected (and dominating), and calculates all cut vertices $S \subset V(G)$ in linear time using Tarjan's algorithm~\cite{Ta72}. The algorithm then recursively removes a vertex from $V(G) \setminus S$ in all possible ways (such that the resulting set is still dominating in case of $N_{dom}(G)$). The resulting set is again connected and the enumeration process continues for each generated set until either the empty set is reached or the current set has been encountered before. The algorithm uses memoization with hashing to efficiently prevent the algorithm from double counting the same set.

\begin{algorithm}[ht!]
\caption{RecursivelyGenerate(Graph $G$, Vertex subset $V'$, Bool \textit{Dominating})}
\label{algo:recursivelyCount}
  \begin{algorithmic}[1]
            \STATE // This function recursively generates all vertex sets $V'$ such that the graph induced by $V'$ is connected.\\
            \STATE // If \textit{Dominating} is \textit{True}, it will only generate dominating connected induced subgraphs.
            \IF{The function was not called before with parameters $G$, $V'$ and \textit{Dominating}}
               \STATE Output $V'$ // The set $V'$ is connected (and dominating if \textit{Dominating} is \textit{True})
               \STATE Compute all cut vertices $S$ of the graph induced by $V'$
               \FOR{each vertex $u \in V' \setminus S$}
                    \IF{\textit{Dominating}}
                       \IF{Every neighbor of $u$ which is not in $V'$ has some neighbor in $V' \setminus \{u\}$}
                            \STATE RecursivelyGenerate($G$, $V' \setminus \{u\}$, \textit{Dominating})
                       \ENDIF
                    \ELSE
                       \STATE RecursivelyGenerate($G$, $V' \setminus \{u\}$, \textit{Dominating})
                    \ENDIF
		    \ENDFOR
            \ENDIF	  
  \end{algorithmic}
\end{algorithm}

\subsection*{Determining coefficient matrices: lower bounds for $c_d$ and $\ct_d$}

We also implemented two algorithms for constructing the coefficient matrices that express the recursion relations for the quantities involving tiles. The largest matrices considered for this paper, corresponding to the states of the gadgets from Section~\ref{sec:lowerbounds}, have several million entries. The two algorithms generate all states by assigning a type to each cell in all possible ways that represent a valid state (e.g. a fixed occupied cell cannot be adjacent to a loose occupied cell in the same column). The first algorithm merges states that are isomorphic to each other, whereas the second algorithm does not take symmetry into account. The algorithms then fill each entry of the matrix by checking whether the state corresponding to that row can be preceded by the state corresponding to that column. Finally, the algorithm computes the largest eigenvalues of the matrix and computes the constant for the lower bounds on $c_d$ (or $\ct_d$) by taking the $n$-th root, where $n$ is the order of the gadget.

The code and data related to this paper have been made publicly available at~\cite{CGJ23}.

\section{Table summarizing the computations of $c_{4,g}(n)$ and $\ct_{4,g}(n)$}\label{sec:app2}

\begin{table}[h]
    \centering
    \begin{tabular}{|c| c | c c | c c |}
\hline
$n$ & $g$ & $c_{4,g}(n)$ & \text{g. ex.} & $\ct_{4,g}(n)$ & \text{g. ex.} \\
\hline
5 & 3 & \textbf{1.9873} & 3 & 1.9873 & 3\\
6 & 3 & 1.9786 & 3 & 1.9442 & 3\\
7 & 3 & 1.9672 & 3 & 1.9442 & 3\\
8 & 3 & \textbf{1.9766} & 4 & 1.9680 & 4\\
9 & 3 & 1.9590 & 3& 1.9320 & 3\\
10 & 3 & 1.9603 & 4& 1.9470 & 4\\
11 & 3 & 1.9515 & 4& 1.9256 & 4\\
12 & 3 & 1.9485 & 4& 1.9301 & 4\\
13 & 3 & 1.9459 & 4& 1.9196 & 4\\
14 & 3 & 1.9433 & 4& 1.9204 & 4\\
15 & 3 & 1.9412 & 4& 1.9161 & 4\\
16 & 3 & 1.9406 & 4& 1.9134 & 4\\
17 & 3 & 1.9383 & 4& 1.9111 & 4\\
19 & 5 & 1.9369 & 5& 1.9001 & 5\\
20 & 5 & 1.9358 & 5&  1.8996 & 5\\
21 & 5 & 1.9349 & 5& 1.8988 & 5\\
22 & 5 & 1.9342 & 5& 1.8988 & 5\\
23 & 5 & 1.9335 & 5&  1.8984 & 5\\
24 & 5 & 1.9330 & 5& 1.8982 & 5\\
26 & 6 & \textbf{1.9321} & 6&  1.8945 & 6\\
28 & 6 & 1.9314 & 6&  1.8942 & 6\\
\hline
\end{tabular}
    \caption{Summary of the computations of $c_{4,g}(n)$,  $\ct_{4,g}(n)$ (rounded up to $4$ decimals) and the girths of the extremal graphs (g. ex.)}
    \label{tab:c_{4,g}(n)}
\end{table}

\end{document}